\documentclass[12pt]{amsart}
\usepackage{amsmath}
\usepackage{amsthm,amssymb}
\usepackage{enumerate}
\usepackage[utf8]{inputenc}
\usepackage{hyperref}
\theoremstyle{plain}
\newtheorem{theorem}{Theorem}[section]
\newtheorem{lemma}[theorem]{Lemma}
\newtheorem{prop}[theorem]{Proposition}

\theoremstyle{definition}

\begin{document}
	\title[Images of multilinear polynomials]{Images of multilinear polynomials\\ on generalized quaternion algebras}
		
\author[P. V. Danchev]{Peter V. Danchev $^{1,*}$}
\author[T. H. Dung]{Truong Huu Dung $^{2}$}
\author[T. N. Son]{Tran Nam Son $^{3,4}$}

\address{[1] Institute of Mathematics and Informatics, Bulgarian Academy of Sciences, 1113 Sofia, Bulgaria}
\address{[2] Department of Mathematics, Dong Nai University, 9 Le Quy Don Str., Tan Hiep Ward, Bien Hoa City, Dong Nai Province, Vietnam}
\address{[3] Faculty of Mathematics and Computer Science, University of Science, Ho Chi Minh City, Vietnam}
\address{[4] Vietnam National University, Ho Chi Minh City, Vietnam}

\email{\newline Peter V. Danchev: danchev@math.bas.bg or pvdanchev@yahoo.com; \newline
	Truong Huu Dung: thdung@dnpu.edu.vn or dungth0406@gmail.com; \newline
	Tran Nam Son: trannamson1999@gmail.com}
	
	\keywords{Quaternion algebras, multilinear polynomial evaluations, Kaplansky conjecture \\
		\protect \indent 2020 {\it Mathematics Subject Classification.} 14A22; 16H05;  16K20\\ \protect \indent * Corresponding author:  Peter V. Danchev}
	
    \maketitle
    
	\begin{abstract}
The main goal of this paper is to extend [J. Algebra Appl. 20 (2021), 2150074] to generalized quaternion algebras, even when these algebras are not necessarily  division rings. More precisely, in such cases, the image of a multilinear polynomial evaluated on a quaternion algebra is a vector space and we additionally provide a classification of possible images.
	\end{abstract}

\section{Introduction}

Throughout of the body of this article, let $m$ be an integer greater than $1$ and let $F\langle X\rangle$ be the free associative algebra over $F$ which is freely generated by the set $X=\{x_1,x_2,\ldots,x_m\}$ of $m$ noncommutative variables. We refer to elements of $F\langle X\rangle$ as \textit{polynomials}. A polynomial $p\in F\langle X\rangle$ is called \textit{multilinear} of degree $m$ if it is of the form: $$p(x_1,\ldots,x_m)=\sum_{\sigma\in S_m}\lambda_\sigma x_{\sigma(1)}\cdots x_{\sigma(m)}$$ in which $S_m$ is the symmetric group in $m$ letters and the coefficients $\lambda_\sigma$ are constants in $F$. For any $F$-algebra $A$ and $p\in F\langle X\rangle$, we set $p(A)=\{p(a_1,\ldots,a_m)\mid a_1,\ldots,a_m\in A\}$ which is called the \textit{image} of $p$ evaluated on $A$. The question which subsets of the matrix ring over $F$ are images of polynomials  was reputedly raised by Kaplansky according to \cite{Pa_KaMaRo_12}. On the other hand, in \cite[Entry 1.98]{Bo_FiKhSh_93}, Lvov conjectured that the image of a multilinear polynomial evaluated on the matrix ring over a field is a vector space. A special case on polynomials of degree two has been known for a long time (see \cite{Pa_AlMu_57,Pa_Sho_36}), which is extended by Mesyan \cite{Pa_Zak_13} to multilinear polynomial of small degree. Meanwhile, in 2012, Kanel-Belov, Malev, and Rowen \cite{Pa_KaMaRo_12} gave a complete description in case the algebra of $2\times 2$ matrices over a quadratically closed field and with the partial solution after that. In spite of many efforts, however, conjectures seems to be far from being resolved. Some variations have been studied extensively. For example, the images on Lie algebras \cite{Pa_AnEm_15} as well as Jordan algebras \cite{Pa_Ma_22} have been discussed. For the most recent results on images of polynomials, we recommend the survey \cite{Pa_KaMaRo_20}.

Let $F$ be a field. When the characteristic of $F$ is not $2$, a \textit{quaternion algebra} over $F$ is a ring that is a $4$-dimensional vector space over $F$ with a basis $1,i,j,k$ with the following multiplicative relations: $i^2\in F$ and $j^2\in F$ are nonzero and $k=ij=-ji$ and every $c\in F$ commutes with $i$ and $j$. In case of the characteristic of $F$ is $2$, a \textit{quaternion algebra} over $F$ is similarly defined as a ring as well as a $4$-dimension vector space over $F$ with relations: $j^2\in F$ is nonzero and $i^2+i\in F$ and $k=ij=j(i+1)$. Moreover, $\mathbb{H}_F$ is called a \textit{division quaternion algebra} if $\mathbb{H}_F$ is a division ring. Let us see more in \cite{Bo_Vo_21}. In 2021, Malev \cite{Pa_Ma_21} gave a complete description of the images of multilinear polynomials on the division quaternion algebra over the field of real numbers. In the present paper, we will extend the work of Malev \cite{Pa_Ma_21} to certain quaternion algebras.  More precisely, let $\mathbb{H}_F$ be a quaternion algebra over a field $F$ and let $p\in F\langle X\rangle$ be multilinear, if $p(\mathbb{H}_F)$ is neither $\{0\}$ nor $F$, then we can classify the possible cases:
\medskip

\noindent\textbf{Case 1.} $\mathbb{H}_F$ is not a division ring. Then, $p(\mathbb{H}_F)$ contains the set (resp. $\{\alpha i+\beta j+\gamma k\mid \alpha,\beta,\gamma \in F\}$) $\{\alpha j^2+\beta j+\gamma k\mid \alpha,\beta,\gamma \in F\}$ if the characteristic of $F$ is (resp. not) $2$.
\medskip

\noindent\textbf{Case 2.} $\mathbb{H}_F$ is  a division ring. If characteristic of $F$ is not $2$, then $p(\mathbb{H}_F)\in\{\{0\}, F, \{\alpha i+\beta j+\gamma k\mid \alpha,\beta,\gamma \in F\}, \mathbb{H}_F\}$ in either of the following cases:
\begin{enumerate}[\rm (i)]
	\item $F$ is a quadratically closed field,  that is, a field in which every element has a square root
	\item $i^2=j^2=-1$ and $F$ is a Pythagorean field, that is, a field in which every sum of two squares is a square.
\end{enumerate} These statements are mentioned in Section 2. In sections 3 and 4, using this result, we  investigate products of images of multilinear polynomials which is initiated by \cite{Pa_DaDuSo_23,Pa_DuSo} and trace vanishing polynomials of quaternion algebras which is an analog version in \cite{Pa_KaMaRo_16}.

\section{Quaternion algebras}

First, we together make some interesting observations. According to \cite[Main Theorem 5.4.4 and Theorem 6.4.11]{Bo_Vo_21}, if the quaternion algebra $\mathbb{H}_F$ over a field $F$ is not a division ring, then $\mathbb{H}_F$ is isomorphic to the matrix ring $\mathrm{M}_2(F)$ over $F$ as an $F$-algebra. According to \cite[Theorem 1]{Pa_Ma_14}, if $p\in F\langle X\rangle$ is multilinear, then $p(\mathrm{M}_2(F))$ is either $\{0\}$, or $F1$ (the set of scalar matrices), or $p(A)$ contains the set $\mathrm{sl}_2(F)$ of matrices whose traces are zero. In particular, if $F$ is the field of real numbers, then $p(A)\in\{\{0\}, F1,\mathrm{sl}_2(F),\mathrm{M}_2(F)\}$, which is the same result if $F$ is a quadratically closed field (see \cite{Pa_KaMaRo_12}). It is well known by \cite{Pa_AlMu_57,Pa_Sho_36} that $\mathrm{sl}_2(F)$ coincides with $s_2(\mathrm{M}_2(F))$ where $s_2$ stands for $s_2=x_1x_2-x_2x_1$ is a polynomial of degree two in $F\langle x_1,x_2\rangle$ which is frequently called the \textit{standard} polynomial of degree two. On the other hand, $F1$ is itself the center of $\mathrm{M}_2(F)$ which is frequently denoted by $Z(\mathrm{M}_2(F))$. Hence, by further using the result of \cite{Pa_Ma_21}, we obtain the following result:

\begin{prop}\label{not division}
	Let $\mathbb{H}_F$ be a quaternion algebra over a field $F$ and let $p\in F\langle X\rangle$ be multilinear. 
	\begin{enumerate}[\rm (i)]
		\item If $\mathbb{H}_F$ is not a division ring, then $p(\mathbb{H}_F)$ is either $\{0\}$, or $F$, or $p(\mathbb{H}_F)$ contains $s_2(\mathbb{H}_F)$. In particular, if $F$ is a quadratically closed field, then $p(\mathbb{H}_F)\in\{\{0\}, F,s_2(\mathbb{H}_F),\mathbb{H}_F\}$.
		\item If $F$ is the field of real numbers, then $p(\mathbb{H}_F)\in\{\{0\},F,s_2(\mathbb{H}_F),\mathbb{H}_F\}$.
	\end{enumerate}
\end{prop}

The case that $\mathbb{H}_F$ is a division quaternion algebra over a field $F$ will be discussed right now. We  first consider multilinear polynomials of small degree. From the technique and strategy of \cite{Pa_CoWa_16} and \cite{Pa_Zak_13}, it follows immediately the following results, so it will be skipped.

\begin{prop}
	Let $\mathbb{H}_F$ be a quaternion algebra over a field $F$ and let $p\in F\langle X\rangle$ be multilinear.
	\begin{enumerate}[\rm (i)]
		\item If $p\in F\langle x_1,x_2\rangle$, then $p(\mathbb{H}_F)\in\{\{0\}, s_2(\mathbb{H}_F), \mathbb{H}_F\}$.
		\item If $p\in F\langle x_1,x_2,x_3\rangle$, then either $p(\mathbb{H}_F)\in\{\{0\}, s_2(\mathbb{H}_F), \mathbb{H}_F\}$ or $p$ has the form: $p=\lambda_1s_2(x_1,s_2(x_3,x_2))+\lambda_2s_2(x_3,s_2(x_1,x_2))$ where $\lambda_1,\lambda_2\in F$.
		\item If $p\in F\langle x_1,x_2,x_3,x_4\rangle$, then $p$ has the form:
		\begin{eqnarray*}
			p&=&\lambda_1s_2(s_2(s_2(x_2,x_1),x_3),x_4)+\lambda_2s_2(s_2(s_2(x_3,x_1),x_2),x_4)\\
			&+&\lambda_3s_2(s_2(s_2(x_4,x_1),x_2),x_3)+\lambda_4s_2(x_1,x_2)s_2(x_3,x_4)\\
			&+&\lambda_5s_2(x_1,x_3)s_2(x_2,x_4)+\lambda_6s_2(x_1,x_4)s_2(x_2,x_3)\\
			&+&\lambda_7s_2(x_2,x_3)s_2(x_1,x_4)+\lambda_8s_2(x_2,x_4)s_2(x_1,x_3)\\
			&+&\lambda_9s_2(x_3,x_4)s_2(x_1,x_2).
		\end{eqnarray*}
	\end{enumerate}
\end{prop}

Note that if $p=\lambda_1s_2(x_1,s_2(x_3,x_2))+\lambda_2s_2(x_3,s_2(x_1,x_2))$ and the characteristic of $F$ is $0$, then by using \cite[Theorem 3.4]{Pa_LiTsui_16}, it follows that $p(\mathbb{H}_F)$ contains elements of reduced trace $0$. From these observations, the next work we can think of is studying about $s_2(\mathbb{H}_F)$. On the other hand, if $p\in F\langle X\rangle$ is multilinear, then $$p(x_1,\ldots,x_m)=\sum_{\sigma\in S_m}\lambda_\sigma x_{\sigma(1)}\cdots x_{\sigma(m)}$$ in which $S_m$ is the symmetric group in $m$ letters and the coefficients $\lambda_\sigma$ are constants in $F$. From \cite[Lemma 3.3]{Pa_GaMe_22}, it follows that
\begin{enumerate}[\rm (i)]
	\item If $\sum_{\sigma\in S_m}\lambda_\sigma\neq0$, then by  $p(\mathbb{H}_F)=\mathbb{H}_F$.
	\item $\sum_{\sigma\in S_m}\lambda_\sigma=0$ if and only if $p$ belongs to the T-ideal $\langle s_2(x_1,x_2)\rangle^T$ generated by the polynomial $s_2$.
\end{enumerate} Hence, the investigation of $s_2(\mathbb{H}_F)$ play an important role.

To do this, we need a series technical claims as follows and recall that  a Pythagorean field is a field in which every sum of two squares is a square. Let us see more in \cite{Bo_Ra_93}.

\begin{lemma}\label{s2-1}
	Let $F$ be a field. If $a,b,c$ belong to $F$, then there exist $x_1,x_2,x_3\in F$ such that
	$$\begin{cases}
		x_1^2+x_2^2+x_3^2=1,\\
		ax_1+bx_2+cx_3=0
	\end{cases}$$ unless $a^2+b^2\neq0$ and the characteristic of $F$ is not $2$. This statement remains valid in case that $a^2 + b^2 \neq 0$ and the characteristic of $F$ is not $2$, especially when $F$ is a Pythagorean field.
\end{lemma}

\begin{proof}
	Let $a,b,c\in F$. If $a=b=c=0$, then we can choose $x_1=1$ and $x_2=x_3=0$. Otherwise, without loss of generality, we can assume $a\neq0$.	We divide the proof into the following cases:
	\medskip
	
	\noindent\textbf{Case 1.} $a^2+b^2=0$. Then, since $a$ is nonzero, $b$ must be nonzero.  If $c=0$, then we choose $x_1=\frac{-b}{a}$ and $x_2=x_3=1$. Now, let us examine the case when $c\neq0$. If the characteristic of $F$ is not $2$, then we can choose $x_1=\frac{-c}{2a}, x_2=\frac{-c}{2b},$ and $x_3=1$. Next, we consider the characteristic of $F$ being $2$. Within this subcase, since $a^2+b^2=0$, the sum $a+b$ of $a$ and $b$ must be $0$, that is $a=b$. If there exist $x_1,x_2,$ and $x_3$ in $F$ satisfying $x_1^2+x_2^2+x_3^2=1$ and $ax_1+bx_2+cx_3=0$, then $x_1+x_2+x_3=1$ and $a+c\neq0$ implying $x_3=\frac{a}{a+c}$ and $x_1+x_2=\frac{c}{a+c}$,  and thus we can choose $x_1=0$ and $x_2=\frac{c}{a+c}$.
	\medskip
	
	\noindent\textbf{Case 2.} $a^2+b^2\neq0$. If the characteristic of $F$ is $2$, then $a+b\neq0$ and we can choose $x_1=\frac{-b}{a+b}, x_2=\frac{a}{a+b},$ and $x_3=0$. Continuing ahead, we focus on the final subcase, where the characteristic of $F$ is not $2$. Moreover, within this subcase, we make an additional assumption that $F$ is a Pythagorean field. Then, there exists $d\in F$ such that $d^2=a^2+b^2$, so we can choose $x_1=\frac{-b}{d}, x_2=\frac{a}{d},$ and $x_3=0$.
\end{proof}

\begin{lemma}\label{s2-2}
	Let $F$ be a field. If $a,b,c$ belong to $F$, then there exist $x_1,x_2,x_3,x_4,x_5,x_6\in F$ such that
	$$\begin{cases}
		x_2x_6-x_3x_5=a,\\
		x_3x_4-x_1x_6=b,\\
		x_1x_5-x_2x_4=c,
	\end{cases}$$ except when $a^2+b^2\neq0$ and the characteristic of $F$ is not $2$. This statement remains valid in case that $a^2 + b^2 \neq 0$ and the characteristic of $F$ is not $2$, especially when $F$ is a Pythagorean field.
\end{lemma}

\begin{proof}
	Let $a,b,c\in F$. Using Lemma~\ref{s2-1}, there exist $x_1,x_2,x_3\in F$ such that $x_1^2+x_2^2+x_3^2=1$ and $ax_1+bx_2+cx_3=0$ unless $a^2+b^2\neq0$ and the characteristic of $F$ is not $2$. In case that $a^2 + b^2 \neq 0$ and the characteristic of $F$ is not $2$, we need to add the assumption that $F$ is a Pythagorean field. The proof is completed by choosing $x_4=bx_3-x_2c, x_5=x_1c-ax_3,$ and $x_6=ax_2-x_1b$.
\end{proof}

We are now in a position to prove the following theorem for standard polynomials of degree $2$.

\begin{theorem}\label{s2}
	Let $\mathbb{H}_F$ be a quaternion algebra over a field $F$.
	\begin{enumerate}[\rm (i)]
		\item If the characteristic of $F$ is not $2$, then 
		$s_2(\mathbb{H}_F)$ is contained in $$\{\alpha i+\beta j+\gamma k\mid \alpha,\beta,\gamma\in F\}.$$ In particular, if $F$ is a Pythagorean field, then $$s_2(\mathbb{H}_F)=\{\alpha i+\beta j+\gamma k\mid \alpha,\beta,\gamma\in F\}=s_2(s_2(\mathbb{H}_F)).$$ Here $s_2(s_2(\mathbb{H}_F))=\{s_2(a,b)\mid a,b\in s_2(\mathbb{H}_F)\}$.
		\item If the characteristic of $F$ is $2$, then $$s_2(\mathbb{H}_F)=\{\alpha j^2+\beta j+\gamma k\mid \alpha,\beta,\gamma\in F\}=s_2(s_2(\mathbb{H}_F)).$$
	\end{enumerate} In particular,  $s_2(\mathbb{H}_F)$ is a vector space over $F$ in either of the following cases:
	\begin{enumerate}[\rm (1)]
	\item The characteristic of $F$ is not $2$ and $F$ is a Pythagorean field.
	\item The characteristic of $F$ is $2$.
\end{enumerate}
\end{theorem}

\begin{proof}
	Let $\alpha,\beta\in\mathbb{H}_F$, namely $\alpha=\alpha_1+\alpha_2 i+\alpha_3 j+\alpha_4$ and $\beta=\beta_1+\beta_2 i+\beta_3 j+\beta_4 k$ for some $\alpha_1,\alpha_2,\alpha_3,\alpha_4,\beta_1,\beta_2,\beta_3,\beta_4\in F$. Then, since the multilinearity  of $s_2$, direct calculation shows that \begin{eqnarray*}
		&& s_2(\alpha,\beta)\\
		&=& s_2(\alpha_1+\alpha_2 i+\alpha_3 j+\alpha_4 k,\beta_1+\beta_2 i+\beta_3 j+\beta_4 k)\\
		&=&\alpha_1\beta_1 s_2(1,1)+\alpha_1\beta_2 s_2(1,i)+\alpha_1\beta_3 s_2(1,j)+\alpha_1\beta_4 s_2(1,k)\\
		&+&\alpha_2\beta_1 s_2(i,1)+\alpha_2\beta_2p(i,i)+\alpha_2\beta_3 s_2(i,j)+\alpha_2\beta_4 s_2(i,k)\\
		&+&\alpha_3\beta_1 s_2(j,1)+\alpha_3\beta_2 s_2(j,i)+\alpha_3\beta_3 s_2(j,j)+\alpha_3\beta_4 s_2(j,k)\\
		&+&\alpha_4\beta_1 s_2(k,1)+\alpha_4\beta_2 s_2(k,i)+\alpha_4\beta_3 s_2(k,j)+\alpha_4\beta_4 s_2(k,k)\\
		&=&\alpha_2\beta_3 s_2(i,j)+\alpha_2\beta_4 s_2(i,k)+\alpha_3\beta_2 s_2(j,i)+\alpha_3\beta_4 s_2(j,k)\\
		&+&\alpha_4\beta_2 s_2(k,i)+\alpha_4\beta_3 s_2(k,j)
	\end{eqnarray*} We divide the proof into two cases as follows:
	\medskip
	
	\noindent\textbf{Case 1.} The characteristic of $F$ is not $2$. Observe that $ s_2(i,k)=-2j, s_2(k,i)=2j, s_2(j,i)=-2k, s_2(i,j)=2k, s_2(j,k)=2i, s_2(k,j)=-2i.$ Hence, $$ s_2(\alpha,\beta)=2(\alpha_3\beta_4-\alpha_4\beta_3)i+2(\alpha_4\beta_2-\alpha_2\beta_4)j+2(\alpha_2\beta_3-\alpha_3\beta_2)k.$$ This means that $s_2(\mathbb{H}_F)$ is contained in $\{\alpha i+\beta j+\gamma k\mid \alpha,\beta,\gamma\in F\}$. 
	\medskip
	
	Now, we add the assumption that $F$ is a Pythagorean field. Let $x=\alpha i+\beta j+\gamma k$ where $\alpha,\beta,\gamma\in F$. Using Lemma~\ref{s2-2}, there exist $x_1,x_2,x_3,x_4,x_5,x_6\in F$ such that
	$$\begin{cases}
		x_2x_6-x_3x_5=\frac{\alpha}{2},\\
		x_3x_4-x_1x_6=\frac{\beta}{2},\\
		x_1x_5-x_2x_4=\frac{\gamma}{2},
	\end{cases}$$ so $x=s_2\left(x_1i+x_2j+x_3k,x_4i+x_5j+x_6k\right)\in s_2(\mathbb{H}_F)$. Hence, $s_2(\mathbb{H}_F)$ contains $\{\alpha i+\beta j+\gamma k\mid \alpha,\beta,\gamma\in F\}$. Therefore, $$s_2(\mathbb{H}_F)=\{\alpha i+\beta j+\gamma k\mid \alpha,\beta,\gamma\in F\}.$$ On the other hand, it is obvious that $s_2(s_2(\mathbb{H}_F))$ is contained in $s_2(\mathbb{H}_F)$. From the above proof, if $x=\alpha i+\beta j+\gamma k\in s_2(\mathbb{H}_F)$ where $\alpha,\beta,\gamma\in F$, then $x=s_2\left(x_1i+x_2j+x_3k,x_4i+x_5j+x_6k\right)\in s_2(s_2(\mathbb{H}_F))$ for some $x_1,x_2,x_3,x_4,x_5,x_6\in F$. Hence, $s_2(\mathbb{H}_F)=s_2(s_2(\mathbb{H}_F))$.
	\medskip
	
	\noindent\textbf{Case 2.} The characteristic of $F$ is $2$. Similarly, observe that $ s_2(i,k)=k, s_2(k,i)=-k, s_2(j,i)=-j, s_2(i,j)=j, s_2(j,k)=-j^2, s_2(k,j)=j^2.$ Therefore, $$ s_2(\alpha,\beta)=(\alpha_4\beta_3-\alpha_3\beta_4)j^2+(\alpha_2\beta_3-\alpha_3\beta_2)j+(\alpha_2\beta_4-\alpha_4\beta_2)k.$$ We deduce that $\{\alpha j^2+\beta j+\gamma k\mid \alpha,\beta,\gamma\in F\}$ contains $s_2(\mathbb{H}_F)$. On the other hand, repeating the similar arguments in Case 1, if $x=\alpha j^2+\beta j+\gamma k$ where $\alpha,\beta,\gamma\in F$, then by Lemma~\ref{s2-2},  there exist $x_1,x_2,x_3,x_4,x_5,x_6\in F$ such that
	$$\begin{cases}
		x_2x_6-x_3x_5=-\alpha,\\
		x_3x_4-x_1x_6=-\gamma,\\
		x_1x_5-x_2x_4=\beta,
	\end{cases}$$ so $x=s_2(x_2i+x_2j+x_3k,x_4i+x_5j+x_6k)\in s_2(\mathbb{H}_F)$. This means that $s_2(\mathbb{H}_F)$ contains $\{\alpha j^2+\beta j+\gamma k\mid \alpha,\beta,\gamma\in F\}$. Therefore, we can conclude that $s_2(\mathbb{H}_F)=\{\alpha j^2+\beta j+\gamma k\mid \alpha,\beta,\gamma\in F\}$. Again, repeating the arguments in Case 1, we conclude that $s_2(\mathbb{H}_F)=s_2(s_2(\mathbb{H}_F))$.
\end{proof}

From the investigation of standard polynomials of degree $2$, we can easily extract the following non-trivial generalization for multilinear polynomials of degree $2$.

\begin{theorem}\label{2}
	Let $\mathbb{H}_F$ be a quaternion algebra over a field $F$ and let $p\in F\langle x_1,x_2\rangle$ be multilinear. 
	\begin{enumerate}[\rm (i)]
		\item If the characteristic of $F$ is not $2$, then $p(\mathbb{H}_F)$ is contained in $\{\alpha i+\beta j+\gamma k\mid \alpha,\beta,\gamma\in F\}$ unless either $p(\mathbb{H}_F)=\{0\}$ or $p(\mathbb{H}_F)=\mathbb{H}_F$. In particular, if $F$ is a Pythagorean field, then $$p(\mathbb{H}_F)\in\{\{0\}, \{\alpha i+\beta j+\gamma k\mid \alpha,\beta,\gamma\in F\}, \mathbb{H}_F\}.$$
		\item If the characteristic of $F$ is $2$, then $$p(\mathbb{H}_F)\in\{\{0\}, \{\alpha j^2+\beta j+\gamma k\mid \alpha,\beta,\gamma\in F\}, \mathbb{H}_F\}.$$
	\end{enumerate} In particular, $p(\mathbb{H}_F)$ is a vector space over $F$ in either of the following cases:
	\begin{enumerate}[\rm (1)]
	\item The characteristic of $F$ is not $2$ and $F$ is a Pythagorean field.
	\item The characteristic of $F$ is $2$.
	\end{enumerate}
\end{theorem}

In case of multilinear polynomials of degree three and four, we also obtain a description as follows.

\begin{theorem}\label{34}
	Let $\mathbb{H}_F$ be a quaternion algebra over a field $F$ and let $p\in\langle X\rangle$ be multilinear. If The degree of $p$ is either three or four, then there are possible cases:
	\begin{enumerate}[\rm (i)]
		\item $p(\mathbb{H}_F)=\{0\}$.
		\item $p(\mathbb{H}_F)=\mathbb{H}_F$.
		\item If the characteristic of $F$ is not $2$, then $p(\mathbb{H}_F)$ is contained in $\{\alpha i+\beta j+\gamma k\mid \alpha,\beta,\gamma\in F\}$.
		\item If the characteristic of $F$ is $2$, then  $\{\alpha j^2+\beta j+\gamma k\mid \alpha,\beta,\gamma\in F\}$ contains $p(\mathbb{H}_F)$.
	\end{enumerate}
\end{theorem}

On the other hand, from Theorem~\ref{s2}, it follows immediately to know images of multilinear polynomials defined as follows. Let $m=2^k$ for some a positive integer $k$. We define a polynomial $v_k(x_1,\ldots,x_{2^k})\in F\langle X\rangle$ successively as follows: set $v_1(x_1,x_2)=x_1x_2-x_2x_1$, assume that $v_{k-1}(x_1,\ldots,x_{2^{k-1}})$ is defined, and then put $$v_k(x_1,\ldots,x_{2^k})=v_1(v_{k-1}(x_1,\ldots,x_{2^{k-1}}),v_{k-1}(x_{2^{k-1}+1},\ldots,x_{2^k})).$$ This polynomial relates to the solvability of a lie algebra.

\begin{theorem}
	Let $\mathbb{H}_F$ be a quaternion algebra over a field $F$. For each positive integer $k$, the image of $v_k$ evaluated on $\mathbb{H}_F$ coincides with $s_2(\mathbb{H}_F)$ in cases of the following:
	\begin{enumerate}[\rm (i)]
		\item The characteristic of $F$ is not $2$ and $F$ is a Pythagorean field.
		\item The characteristic of $F$ is $2$.
	\end{enumerate} In particular, in both cases, $v_k(\mathbb{H}_F)$ is a vector space over $F$ and $$s_2(s_2(\mathbb{H}_F))=s_2(\mathbb{H}_F)=v_1(\mathbb{H}_F)=v_2(\mathbb{H}_F)=\cdots=v_k(\mathbb{H}_F)=\cdots.$$
\end{theorem}

\begin{proof}
	We prove theorem by induction on $k$. The statement is trivial in case $k=1$ by using Theorem~\ref{s2}. Assume that $k=2$. According to Theorem~\ref{s2}, it follows that $s_2(\mathbb{H}_F)=s_2(s_2(\mathbb{H}_F))$, so if $a\in s_2(\mathbb{H}_F)$, then \begin{eqnarray*}
		a&=&s_2(b_1,b_2)\\
		&=&s_2(s_2(a_1,a_2),s_2(a_3,a_4))\\
		&=&v_1(v_1(a_1,a_2),v_1(a_3,a_4))\\
		&=&v_2(a_1,a_2,a_3,a_4)\in v_2(\mathbb{H}_F)
	\end{eqnarray*} for some $b_1,b_2,a_1,a_2,a_3,a_4\in s_2(\mathbb{H}_F)$ in which $b_1=s_2(a_1,a_2)$ and $b_2=s_2(a_3,a_4)$. Hence, $v_2(\mathbb{H}_F)=v_1(\mathbb{H}_F)=s_2(\mathbb{H}_F)=s_2(s_2(\mathbb{H}_F))$. Assume that the statement is true for $k-1$. Now, by the induction hypothesis, if $a\in s_2(\mathbb{H}_F)$, then \begin{eqnarray*}
	a&=&s_2(b_1,b_2)\\
	&=&v_1(b_1,b_2)\\
	&=&v_1(v_{k-1}(a_1,\ldots,a_{2^{k-1}}),v_{k-1}(a_{2^{k-1}+1},\ldots,a_{2^k}))\\
	&=&v_k(a_1,\ldots,a_{2^k})\in v_k(\mathbb{H}_F)
	\end{eqnarray*} for some $b_1,b_2,a_1,\ldots,a_{2^k}\in s_2(\mathbb{H}_F)$ in which $b_1=v_{k-1}(a_1,\ldots,a_{2^{k-1}})$ and $b_2=v_{k-1}(a_{2^{k-1}+1},\ldots,a_{2^k})$. Therefore, $s_2(\mathbb{H}_F)=v_k(\mathbb{H}_F)$. The proof is completed.
\end{proof}

In the rest of this section, as mentioned, we will extend the work of \cite{Pa_Ma_21} to certain division quaternion algebras. Frequently, the quaternion algebra over a field $F$ of characteristic different from $2$ is denoted by $\mathbb{H}(a,b)_F$ when $i^2=-a$ and $j^2=-b$. Recall that  a quadratically closed field is a field in which every element has a square root. Note that every quadratically closed field is a Pythagorean field but not conversely, for example, the field of real numbers is Pythagorean but it is not a quadratically closed field. Let us see more in \cite{Bo_Ra_93}.

In this aspect, we first have the following plain but useful claim.

\begin{lemma} {\rm \cite[Proposition 2.8]{Pa_Fla_01}}\label{Flaut}
	Let $\mathbb{H}(a,b)_F$ be a division quaternion algebra over a field $F$ of characteristic different from $2$. If $\alpha=a_0+a_1i+a_2j+a_3k$ and $\beta=b_0+b_1i+b_2j+b_3k$ belong to $\mathbb{H}(a,b)_F$ where $a_0,a_1,a_2,a_3,b_0,b_1,b_2,b_3\in F$, the necessary and sufficient condition for the linear equation $ax=xb$ has nonzero solutions $x\in\mathbb{H}(a,b)_F$ is $a_0=b_0$ and $aa_1^2+ba_2^2+aba_3^2=ab_1^2+bb_2^2+abb_3^2$.
\end{lemma}

With Lemma~\ref{Flaut} at hand, we obtain the useful result as follows.

\begin{lemma}\label{conjugate}
	Let $\mathbb{H}(a,b)_F$ be a division quaternion algebra over a field $F$ of characteristic different from $2$. If $\alpha=a_0+a_1i+a_2j+a_3k\in\mathbb{H}_F$ where $a_0,a_1,a_2,a_3\in F$, then $x^{-1}\alpha x=a_0+ri$ for some nonzero $x\in\mathbb{H}_F$ and $r\in F$ in either of the following cases:
	\begin{enumerate}[\rm (i)]
		\item $F$ is a quadratically closed field.
		\item $F$ is a Pythagorean field and $a=b=1$.
	\end{enumerate}
\end{lemma}

\begin{proof}
	Let $\alpha=a_0+a_1i+a_2j+a_3k\in\mathbb{H}_F$ where $a_0,a_1,a_2,a_3\in F$. If $F$ is a quadratically closed field, there is $r\in F$ such that $r^2=a_1^2+\frac{b}{a}a_2^2+ba_3^2$. On the other hand, we also obtain a similar result if $F$ is a Pythagorean field and $a=b=1$. Using Lemma~\ref{Flaut}, there exists $x\in\mathbb{H}(a,b)_F$ is nonzero such that $x^{-1}\alpha x=a_0+ri$.
\end{proof}

On the other hand, we also need to know images of multilinear polynomials evaluated on basis quaternions $1,i,j,k$.

\begin{lemma}\label{basic_qua}
	Let $\mathbb{H}_F$ be a division quaternion algebra over a field $F$ and let $p\in F\langle X\rangle$ be multilinear. If $a_1,\ldots,a_m\in\{1,i,j,k\}$, then  $p(a_1,\ldots,a_m)=c\cdot q$ for some $c\in F$ and $q\in\{1,i,j,k\}$.
\end{lemma}

\begin{proof}
	Let $a_1,a_2,\ldots,a_m\in\{1,i,j,k\}$. If $m=2$, then $a_2a_1=\lambda a_1a_2$ for some $\lambda\in F$. Proceeding by induction on $m$, we can conclude that there exists $\lambda\in F$ such that $a_{\sigma(1)}a_{\sigma(2)}\cdots a_{\sigma(m)}=\lambda a_1a_2\cdots a_m$ for every $\sigma\in S_m$. Therefore, $p(a_1,\ldots,a_m)=c\cdot q$ for some $c\in F$ and $q\in\{1,i,j,k\}$.
\end{proof}

We are now to establish the following useful lemma needed for the establishment of the final main result in this section.

\begin{lemma}\label{s2 contain}
	Let $\mathbb{H}(a,b)_F$ be a division quaternion algebra over a field $F$ of characteristic different from $2$ and let $p\in F\langle X\rangle$ be multilinear and non-central. Then, $s_2(\mathbb{H}_F)$ is contained in $p(\mathbb{H}_F)$ in either of the following cases:
	\begin{enumerate}[\rm (i)]
		\item $F$ is a quadratically closed field.
		\item $F$ is a Pythagorean field and $a=b=1$.
	\end{enumerate}
\end{lemma}

\begin{proof}
	Let $x\in s_2(\mathbb{H}_F)$. Using Theorem~\ref{s2} and  Lemma~\ref{conjugate}, $g^{-1}xg=ri$ for some $g\in\mathbb{H}_F$ and $r\in F$. According to Lemma~\ref{basic_qua} and the fact that $p$ is non-central, there exists $y\in p(\mathbb{H}_F)$ which is of the form: $r'\cdot x'$ where $x'\in\{i,j,k\}$ and $r'\in F$ is nonzero. Applying Lemma~\ref{conjugate} again, without loss of generality, we can assume $x'=i$. On the other hand, let us write $y=p(z_1,z_2,\ldots,z_m)$ for some $z_1,z_2,\ldots,z_m\in\mathbb{H}_F$, which implies that $$x=p(grr'^{-1}z_1g^{-1},gz_2g^{-1},\ldots,gz_mg^{-1})\in p(\mathbb{H}_F).$$  Therefore, $s_2(\mathbb{H}_F)$ is contained in $p(\mathbb{H}_F)$.
\end{proof}

We can now summarize all of the above lemmas and  show the validity of the following statement. The idea of the proof is similar in \cite[Theorem 1]{Pa_Ma_21}.

\begin{theorem}\label{division}
	Let $\mathbb{H}(a,b)_F$ be a division quaternion algebra over a field $F$ of characteristic different from $2$ and let $p\in F\langle X\rangle$ be multilinear. Then, $p(\mathbb{H}_F)\in\{\{0\}, F, s_2(\mathbb{H}_F), \mathbb{H}_F\}$ in either of the following cases:
	\begin{enumerate}[\rm (i)]
		\item $F$ is a quadratically closed field.
		\item $F$ is a Pythagorean field and $a=b=1$.
	\end{enumerate}
\end{theorem}

\begin{proof}
	Lemma~\ref{basic_qua} yields that there are four possible cases:
	\begin{description}
		\item[Case 1] $p(\mathbb{H}_F)=\{0\}$.
		\item[Case 2] $p(\mathbb{H}_F)$ is contained in $F$.
		\item[Case 3] $p(\mathbb{H}_F)$ is contained in $s_2(\mathbb{H}_F)$.
		\item[Case 4] $p(\mathbb{H}_F)$ do not belong to $\{\{0\}, F,s_2(\mathbb{H}_F)\}$.
	\end{description} 
	
	Regarding Case 2, we immediately deduce that $p(\mathbb{H}_F)=F$. Using Lemma~\ref{s2 contain}, we conclude that $p(\mathbb{H}_F)=s_2(\mathbb{H}_F)$. Now, let us consider Case 4, and then, there exist $a_1,\ldots,a_m\in\mathbb{H}_F$ such that $p(a_1,\ldots,a_m)\in F$ which is nonzero. Put $$A_1=p(a_1,a_2,\ldots,a_m)$$ which can be view as  a constant polynomial taking only one possible value
	which is a nonzero and for each $i\in\{2,\ldots,m+1\}$, let $$A_i=p(x_1,\ldots,x_{i-1},a_i,\ldots,a_m)$$ be a polynomial in $F\langle x_1,\ldots,x_{i-1}\rangle$. Then, $A_i(\mathbb{H}_F)$ is contained in $A_{i+1}(\mathbb{H}_F)$ for all $i\in\{1,\ldots,m+1\}$ and $A_{m+1}(\mathbb{H}_F)=p(\mathbb{H}_F)$. Hence, there exists $i\in\{1,\ldots,m+1\}$ such that $A_i(\mathbb{H}_F)$ is contained in $F$ and $A_{i+1}(\mathbb{H}_F)$ is not contained in $F$, so there exist $r_1,r_2,\ldots,r_m,r_i^*\in\mathbb{H}_F$ satisfying $p(r_1,\ldots,r_m)\in F$ which is nonzero and $$p(r_1,\ldots,r_{i-1},r_i^*,r_{i+1},\ldots,r_m)$$ which is not in $F$. Let us set $\alpha=p(r_1,\ldots,r_m)$ and $$\beta=p(r_1,\ldots,r_{i-1},r_i^*,r_{i+1},\ldots,r_m).$$ Using Lemma~\ref{conjugate}, $g^{-1}\beta g=a'+r'i$ for some $a'\in F$ and $r'\in F$. Since $\beta$ is not in $F$, the element $r'$ must be nonzero.  Now, let $x\in\mathbb{H}_F$. By using Lemma~\ref{conjugate}, $hxh^{-1}=a+ri$ for some $r\in F$ and $h\in\mathbb{H}_F$ which is nonzero. If $a=0$, then Lemma~\ref{s2 contain} yields that $x\in s_2(\mathbb{H}_F)$ is contained in $p(\mathbb{H}_F)$. Now, we assume $a\neq0$. By putting $r^{**}_i=r^*_i-a'\alpha^{-1}r_i$, direct calculation shows that $$a=g^{-1}ag=p(g^{-1}r_1g,\ldots,g^{-1}r_{i-1}g,g^{-1}a\alpha^{-1}r_ig,g^{-1}r_{i+1}g,\ldots,g^{-1}r_mg)$$ and	$$ri=p(rg^{-1}r_1g,\ldots,rg^{-1}r_{i-1}g,rr'^{-1}g^{-1}r^{**}_ig,rg^{-1}r_{i+1}g,\ldots,rg^{-1}r_mg).$$By putting $$r_i^{***}=h^{-1}(g^{-1}a\alpha^{-1}r_ig+rr'^{-1}g^{-1}r_i^{**}g)h$$ and $r_j^*=h^{-1}rg^{-1}r_jgh$ for all  $j\in\{1,\ldots,m\}$ and $j\neq i$, we can deduce that $$x=p(r_1^*,\ldots,r_{i-1}^*,r_i^{***},r_{i+1}^*,\ldots,r_m^*)\in p(\mathbb{H}_F).$$ Therefore, within Case 4, we conclude that $p(\mathbb{H}_F)=\mathbb{H}_F$.
\end{proof}

Let us see Proposition~\ref{not division} in case of quaternion algebras which are not division rings.

\section{Products of images of multilinear polynomials}

The classical Waring's problem, which is proposed by Edward Waring in 1770 and solved by David Hilbert in 1901, asks whether for every positive integer $k$, there exists a positive integer $g(k)$ such that every positive integer can be expressed as a sum of $g(k)$ $k$th powers of nonnegative integers. Various extensions and variations of this problem have been studied (see \cite{Pa_Fo_02,Pa_Hel_92,Pa_Ka_19,Pa_LaShTi_11,Pa_Sha_09}). In 2020, Bresar initiated the study of various Waring's problems for matrix algebras (see \cite{Pa_Bre_20}), which is recently developed by Bresar and Semrl in \cite{Pa_BreSe_23I,Pa_BreSe_23II}.

In this section, we focus on Waring's problems which asks whether every element of a quaternion algebra can be written as a product of images of multilinear polynomials, which is initiated partially by \cite{Pa_DaDuSo_23,Pa_DuSo}. Recall that a polynomial $p\in F\langle X\rangle$ is an \textit{identity} of an $F$-algebra $A$ if $p(A)=\{0\}$, and $p$ is \textit{central} if $p$ is not an identity of $A$ but $p(A)$ is contained in the center $Z(A)$ of $A$. For convenience, a polynomial $p\in C\langle x_1,\cdots,x_m\rangle$ is called \textit{non-central} if $p$ is not an identity as well as central. For convenience, for subsets $A_1,\ldots,A_n$ of an algebra $A$, we use the symbol $A_1\cdots A_n$ for the set consisting of all products of the form: $a_1\cdots a_n$ where $a_1\in A_1,\ldots,a_n\in A_n$. In case that $A_1,\ldots,A_n$ are all equal to a subset $B$ of $A$, we write $B^n$ instead of $A_1\cdots A_n$.

By using Lemma~\ref{conjugate}, we obtain the following result.

\begin{theorem}\label{p1p2}
	Let $\mathbb{H}(a,b)_F$ be a division quaternion algebra over a field $F$ of characteristic different from $2$ and let $p_1,p_2\in F\langle X\rangle$ be multilinear and non-central. Then, $\mathbb{H}_F=p_1(\mathbb{H}_F)p_2(\mathbb{H}_F)$ in either of the following cases:
	\begin{enumerate}[\rm (i)]
		\item $F$ is a quadratically closed field.
		\item $F$ is a Pythagorean field and $a=b=1$.
	\end{enumerate}
\end{theorem}

\begin{proof}
	Let $\alpha\in\mathbb{H}_F$. According to Lemma~\ref{conjugate}, $h\alpha h^{-1}=x+ri$ for some $x,r\in F$ and $h\in\mathbb{H}_F$ which is nonzero. Assume that $\mathbb{H}_F=\mathbb{H}(a,b)_F$. Then, direct calculation shows that $$\alpha=h^{-1}jh\cdot h^{-1}(-xa^{-1}j+rk)h.$$ Using Theorem~\ref{s2}, both $j$ and $-xa^{-1}j+rk$ belong to $s_2(\mathbb{H}_F)$. On the other hand, Lemma~\ref{s2 contain} yields that both $j$ and $-xa^{-1}j+rk$ respectively belong to $p_1(\mathbb{H}_F)$ and $p_2(\mathbb{H}_F)$, so $h^{-1}jh$ and $h^{-1}(-xa^{-1}j+rk)h$. The proof is completed.
\end{proof}

In the case of quaternion algebras which is not a division ring, we also obtain an interesting result without the assumption that the characteristic of a field is not $2$ as follows.

\begin{theorem}\label{p1p22}
	Let $\mathbb{H}_F$ be a quaternion algebra over a field $F$ and let $p_1,p_2\in F\langle X\rangle$ be multilinear and non-central. If $\mathbb{H}_F$ is not a division ring, then $\mathbb{H}_F=p_1(\mathbb{H}_F)p_2(\mathbb{H}_F)$.
\end{theorem}

\begin{proof}
	If $\mathbb{H}_F$ is not a division ring, then by using  \cite[Main Theorem 5.4.4 and Theorem 6.4.11]{Bo_Vo_21}, it follows that $\mathbb{H}_F$ is isomorphic to the matrix ring $\mathrm{M}_2(F)$ over $F$ as an $F$-algebra. Then, the proof is complete by using \cite[Proposition 5.3]{Pa_DaDuSo_23}.
\end{proof}

Such situations may be difficult in case that quaternion algebras are division rings as well as have characteristic $2$. We will together consider the standard polynomial of degree two. Let $\alpha=bi+cj+dk\in\mathbb{H}_F$ where $b,c,d\in F$. Using Lemma~\ref{s2-2}, we find $x_1,x_2,x_3,x_4,x_5,x_6\in F$ such that $$\begin{cases}
	x_1x_5-x_2x_4=d,\\
	x_3x_4-x_1x_6=c,\\
	x_2x_6-x_3x_5=b,
\end{cases}$$ in which $$\begin{cases}
	x_4=cx_3-x_2d,\\
	x_5=x_1d-bx_3,\\
	x_6=bx_2-x_1c,
\end{cases}$$ so $x_1x_4+x_2x_5+x_3x_6=0$ and $$\alpha=bi+cj+dk=(x_1i+x_2j+x_3k)(x_4i+x_5j+x_6k).$$According to Theorem~\ref{s2}, it follows that $$s_2(\mathbb{H_F})=\begin{cases}
	\{\alpha i+\beta j+\gamma k\mid \alpha,\beta,\gamma \in F\}\text{ if the characteristic of $F$}\\
	\text{is not $2$ and $F$ is a Pythagorean field,}\\\{\alpha j^2 +\beta j+\gamma k\mid \alpha,\beta,\gamma \in F\}\text{ if the characteristic of $F$}\\\text{is $2$}.
\end{cases}$$ Therefore, if the characteristic of $F$ is not $2$ and $F$ is a Pythagorean field, then $s_2(\mathbb{H}_F)$ is contained in $s_2(\mathbb{H}_F)s_2(\mathbb{H}_F)=s_2(\mathbb{H}_F)^2$, and thus $$s_2(\mathbb{H}_F)\subseteq s_2(\mathbb{H}_F)^2\subseteq\cdots\subseteq s_2(\mathbb{H}_F)^k\subseteq\cdots.$$ A natural question is that there is a positive integer $n$ such that $s_2(\mathbb{H}_F)^k=s_2(\mathbb{H}_F)^{k+1}$ for all $k\geq n$, especially $s_2(\mathbb{H}_F)^n=\mathbb{H}_F$. From Theorem~\ref{p1p2} and Theorem~\ref{p1p22}, we conclude that $\mathbb{H}_F=s_2(\mathbb{H}_F)^2$ in either of the following cases:
\begin{enumerate}[\rm (i)]
	\item $\mathbb{H}_F$ is not a division ring.
	\item The characteristic of $F$ is not $2$ and \begin{enumerate}
		\item either $F$ is a quadratically closed field
		\item or $F$ is a Pythagorean field and $i^2=j^2=-1$.
	\end{enumerate}
\end{enumerate}

Moreover,  we can show that if $\mathbb{H}_F$ is a division quaternion algebra over a field $F$ of characteristic $2$ and $$p=s_2(x_1,x_2)+s_2(x_3,x_4)s_2(x_5,x_6)\in F\langle x_1,\ldots,x_6\rangle,$$ then $p(\mathbb{H}_F)=\mathbb{H}_F$.  Indeed, let $\alpha\in\mathbb{H}_F$. According to Theorem~\ref{s2}, since $s_2(\mathbb{H}_F)=s_2(s_2(\mathbb{H}_F))$, it  follows that there exist nonzero $y$ and $z$ which belong to $s_2(\mathbb{H}_F)$ such that $s_2(y,z)$ is nonzero, so 
\begin{eqnarray*}
	\alpha&=&\alpha s_2(y,z)^{-1}s_2(y,z)\\
	&=&s_2(\alpha s_2(y,z)^{-1}y,z)+s_2(z,\alpha s_2(y,z)^{-1})y\in p(\mathbb{H}_F).
\end{eqnarray*} We also obtain the similar result if $$p=s_2(x_1,x_2)s_2(x_3,x_4)+s_2(x_5,x_6)s_2(x_7,x_8)\in F\langle x_1,\ldots,x_8\rangle$$ by expressing $s_2(\alpha s_2(y,z)^{-1}y,z)$ as $s_2(\alpha s_2(y,z)^{-1}yz^{-1},z)z$.

\section{Trace vanishing multilinear polynomials}

Let $R$ be a ring. Standardly, we use the notations $\mathrm{M}_n(R)$ and $\mathrm{sl}_n(R)$ as the ring of matrices of size $n$ over $R$ and the set of all matrices of $\mathrm{M}_n(R)$ whose trace are zero. Similarly, we also replace the field in free algebra by a commutative ring. According to \cite{Pa_KaMaRo_20}, a polynomial $p\in Z(R)\langle X\rangle$ is called \textit{trace vanishing} of $\mathrm{M}_n(R)$ if $p(\mathrm{M}_n(R))$ is contained in $\mathrm{sl}_n(R)$. For example, it is not difficult to see that a special polynomial, which is of the form: $$s_m=\sum_{\sigma\in S_m}\mathrm{sgn}(\sigma)x_{\sigma(1)}\cdots x_{\sigma(m)}$$ where $\mathrm{sgn}(\sigma)$ is the sign of $\sigma$, is trace vanishing if $m$ is even. In 2016, A. Kanel-Belov, S. Malev, L. Rowen describle all possible images of trace vanishing  polynomials for $3\times 3$ (see \cite[Theorem 3, Theorem 4]{Pa_KaMaRo_16}). On the other hand, in the case of multilinear polynomials of three variables, using \cite[Lemma 14]{Pa_Zak_13}, if $$p=\sum_{\sigma\in S_3}\lambda_\sigma x_{\sigma(1)}x_{\sigma(2)}x_{\sigma(3)}$$ is trace vanishing, then $$\sum_{\sigma\in S_3}\lambda_\sigma=\sum_{\sigma\in A_3}\lambda_\sigma=0$$ where $A_3$ is the alternating group of $S_3$. In view of \cite{Pa_BreKl_11}, $p$ is trace vanising if and only if $p$ is a sum of an identity of $\mathrm{M}_n(R)$ and a sum of commutators in $R\langle X\rangle$. 

In this section, we focus on trace vanishing multilinear polynomials evaluated on quaternion algebras. In the first, we recall the \textit{trace} in quaternion algebras. Let $\mathbb{H}_F$ be a quaternion algebra over a field $F$. For any $\alpha=a_0+a_1i+a_2j+a_3\in\mathbb{H}_F$ where $a_0,a_1,a_2,a_3\in F$, the symbol $\mathrm{trace}(\alpha)$ is denoted by the \textit{trace} of $\alpha$ and defined as $$	\mathrm{trace}(\alpha)=\begin{cases}
2a_0\text{ if the characteristic of $F$ is not $2$},\\a_1\text{ if the characteristic of $F$ is $2$}.
\end{cases}$$ We use the symbol $\mathbb{H}_F^0$  for the set of elements of $\mathbb{H}_F$ whose are trace zero. Note that $\mathbb{H}_F^0$ is a vector space over $F$ and $$\mathbb{H}^0_F=\begin{cases}
\{\alpha i+\beta j+\gamma k\mid \alpha,\beta,\gamma \in F\}\text{ if the characteristic of $F$ is not $2$},\\\{\alpha +\beta j+\gamma k\mid \alpha,\beta,\gamma \in F\}\text{ if the characteristic of $F$ is $2$}.
\end{cases}$$ From Theorem~\ref{s2}, it follows:
\begin{enumerate}
	\item If the characteristic of $F$ is $2$, then $s_2(\mathbb{H}_F)=\mathbb{H}^0_F$.
	\item If the characteristic of $F$ is $2$, then $s_2(\mathbb{H}_F)$ is contained in $\mathbb{H}^0_F$. In particular, if $F$ is a Pythagorean field, then $s_2(\mathbb{H}_F)=\mathbb{H}^0_F$.
\end{enumerate}

 Similarly in \cite{Pa_KaMaRo_16}, a polynomial $p\in F\langle X\rangle$ is called \textit{trace vanishing} of $\mathbb{H}_F$ if $p(\mathbb{H}_F)$ is contained in $\mathbb{H}_F^0$. 
 
 \begin{prop}
 	Let $\mathbb{H}_F$ be a quaternion algebra over a field $F$ and let $p\in F\langle X\rangle$ be multilinear of degree at most four. If $p(\mathbb{H}_F)$ is neither $\{0\}$ nor $\mathbb{H}_F$, then $p$ is trace vanishing. In particular, if $p\in F\langle x_1,x_2\rangle$, then $p(\mathbb{H}_F)=\mathbb{H}_F^0$ in either of the following cases:
 	\begin{enumerate}[\rm (i)]
 		\item The characteristic of $F$ is not $2$ and $F$ is a Pythagorean field.
 		\item The characteristic of $F$ is $2$.
 	\end{enumerate}
 \end{prop}
 
 \begin{proof}
 	This is inferred directly from Theorem~\ref{2} and Theorem~\ref{34}.
 \end{proof}

\begin{prop}
	Let $\mathbb{H}_F$ be a quaternion algebra over a field $F$ and let $p\in F\langle X\rangle$ be multilinear. If $p(\mathbb{H}_F)$ is neither $\{0\}$ nor $F$ as well as $\mathbb{H}_F$ and $p$ is trace vanishing, then $p(\mathbb{H}_F)=\mathbb{H}_F^0$ in either of the following cases:
	\begin{enumerate}[\rm (i)]
		\item $\mathbb{H}_F$ is not a division ring.
		\item $\mathbb{H}_F$ is a division ring and the characteristic of $F$ is not $2$ and \begin{enumerate}[\rm (1)]
			\item either $F$ is a quadratically closed field
			\item or $F$ is a Pythagorean field and $i^2=j^2=-1$.
		\end{enumerate}
	\end{enumerate}
\end{prop}

\begin{proof}
	It follows immediately according to Proposition~\ref{not division} and Theorem~\ref{division}.
\end{proof}

As seen above, the investigation of the image of the standard polynomial $s_2$ of degree two plays an important role, especially we can apply the results of the rings of all $2\times 2$ matrices over fields to quaternion algebras which are not division rings. Therefore, we will end this section with an interesting observation. Now, we shift our attention to matrices. We show that if $p\in Z(R)\langle X\rangle$ is multilinear such that $p(\mathrm{M}_n(R))$ is contained in the union $\{0\}\cup (\mathrm{M}_n(R)\setminus\mathrm{sl}_n(R))$ of $\{0\}$ and the set difference $\mathrm{M}_n(R)\setminus\mathrm{sl}_n(R)$ of $\mathrm{M}_n(R)$ and $\mathrm{sl}_n(F)$ and $p(\mathrm{M}_n(R))\neq\{0\}$, then $p(\mathrm{M}_n(R))$ is contained in $Z(\mathrm{M}_n(R))$. The reason why we get attention to $\mathrm{sl}_n(R)$ is that $\mathrm{sl}_n(R)$ coincides with $s_2(\mathrm{M}_n(R))$ when $R$ is a field and the description of $s_2(\mathrm{M}_n(R))$ play an important role as mentioned above. To prove this result, we need to denote $e_{ij}$ by a matrix unit of $\mathrm{M}_n(R)$  which has only one nonzero entry with value 1 at the position $(i,j)$. The following lemma is pivotal which is already known.

\begin{lemma} {\rm\cite[Lemma 2]{Pa_Le_75}} \label{basis matrix}
	Let $R$ be a ring and let $p\in Z(R)\langle X\rangle$ be multilinear. If $a_1,\ldots,a_m\in R$ and $A_1,\ldots,A_m$ are matrix units, then $p(a_1A_1,\ldots,a_mA_m)$ is either a diagonal matrix or a matrix of the form: $a\cdot e_{ij}$ for some $a\in R$ and $i\neq j$.
\end{lemma}

We are now ready to show the final result as mentioned above.

\begin{theorem}
	Let $R$ be a ring and let $p\in Z(R)\langle X\rangle$ be multilinear such that $p(\mathrm{M}_n(R))$ is contained in  $\{0\}\cup (\mathrm{M}_n(R)\setminus\mathrm{sl}_n(R))$ and $p(\mathrm{M}_n(R))\neq\{0\}$. Then, $p(\mathrm{M}_n(R))$ is contained in $Z(\mathrm{M}_n(R))$.
\end{theorem}

\begin{proof}
	According to Lemma~\ref{basis matrix}, if $a_1,\ldots,a_m\in R$ and $A_1,\ldots,A_m$ are matrix units, then $p(a_1A_1,\ldots,a_mA_m)$ is either a diagonal matrix or a matrix of the form: $a\cdot e_{ij}$ for some $a\in R$ and $i\neq j$. Since $p(\mathrm{M}_n(R))$ is contained in the union $\{0\}\cup (\mathrm{M}_n(R)\setminus\mathrm{sl}_n(R))$, the last possible is eliminated. Since every matrix in $\mathrm{M}_n(R)$ is a linear combination of matrix units over $R$ and $p$ is multilinear, it follows that $p(\mathrm{M}_n(R))$ consists only diagonal matrices. Set $x\in p(\mathrm{M}_n(R))$ and write $x=p(B_1,\ldots,B_m)$ where $B_1,\ldots,B_m\in\mathrm{M}_n(R)$, so $$x=\sum_{i=1}^na_ie_{ii}$$ for some $a_1,\ldots,a_n\in R$. On the other hand, for each $j\in\{2,3,\ldots,n\}$, there exist $b_1^j,b_2^j,\ldots,b_n^j\in R$ such that \begin{eqnarray*}
		&&\sum_{i=1}^nb_i^je_{ii}\\
		&=&p\left((\mathrm{I}_n+e_{1j})B_1(\mathrm{I}_n+e_{1j})^{-1},\ldots,(\mathrm{I}_n+e_{1j})B_n(\mathrm{I}_n+e_{1j})^{-1}\right)\\
		&=&(\mathrm{I}_n+e_{1j})p(B_1,\ldots,B_m)(\mathrm{I}_n+e_{1j})^{-1}\\
		&=&\left(\sum_{i=1}^na_ie_{ii}\right)+(a_je_{1j}-a_1e_{1j}).
	\end{eqnarray*} Hence, $a_1=a_2=\cdots=a_n$, and put $a_1=a$. If $r\in R$, then by using the similar argument, for each $j\in\{2,3,\ldots,n\}$, there exist $c_1^j,\ldots,c_n^j\in R$ such that \begin{eqnarray*}
		&&\sum_{i=1}^nc_i^je_{ii}\\
		&=&p\left((\mathrm{I}_n+re_{1j})B_1(\mathrm{I}_n+r e_{1j})^{-1},\ldots,(\mathrm{I}_n+r e_{1j})B_m(\mathrm{I}_n+r e_{1j})^{-1}\right)\\
		&=&(\mathrm{I}_n+r e_{1j})p(B_1,\ldots,B_m)(\mathrm{I}_n+r e_{1j})^{-1}\\
		&=&\left(\sum_{i=1}^na_ie_{ii}\right)+rae_{1j}-are_{1j}, 
	\end{eqnarray*} that is, $ra=ar$. This means that $a$ belongs to the center of $R$. Therefore, $p(\mathrm{M}_n(R))$ is contained in $Z(\mathrm{M}_n(R))$.
\end{proof}

\end{document}